\documentclass[12pt]{amsart}
\newtheorem{theorem}{Theorem}[section]
\newtheorem{lemma}[theorem]{Lemma}

\newtheorem{proposition}[theorem]{Proposition}
\theoremstyle{definition}
\newtheorem{definition}[theorem]{Definition}
\newtheorem{example}[theorem]{Example}

\usepackage{enumerate}
\usepackage{hyperref}
\newtheorem{remark}[theorem]{Remark}
\usepackage[margin=0.75in]{geometry}
\usepackage{calligra}
\usepackage{xcolor, soul}
\sethlcolor{yellow}
\usepackage{mathtools}

\numberwithin{equation}{section}

\DeclareMathOperator\tr{tr}
\def\N{{\mathbb{N}}}
\def\h{{\mathcal{H}}}

\def\s{{\sigma}}
\def\e{{\eta}}
\def\dj{{D_{2^{-j}}}}
\def\D{{D_{2^{m+j}}}}

\def\p{{\phi}}
\def\vv{{V^{t}}}
\def\a{{(T^{t}_{(k,l)})^{2^{-2j}}}}

\newcommand\numberthis{\addtocounter{equation}{1}\tag{\theequation}}

\newcommand{\be}{\begin{equation}}
	\newcommand{\ee}{\end{equation}}

\begin{document}
	\title{Discrete Calder\'{o}n condition for the twisted wavelet system}
	
	\author{Radha Ramakrishnan}
	\address{Department of Mathematics, Indian Institute of Technology, Madras, India}
	\email{radharam@iitm.ac.in}
	\author{Rabeetha Velsamy}
	\address{Department of Mathematics, Indian Institute of Technology, Madras, India}
	\email{rabeethavelsamy@gmail.com}
	
	
	
	
	\begin{abstract}
	In this paper, we obtain an analogue of the discrete Calder\'{o}n condition and prove that this condition is sufficient for an orthonormal twisted wavelet system to be complete in $L^{2}(\mathbb{R}^{2})$.
	\end{abstract}
	\maketitle
	\section{Introduction}
	It is well known in literature that if $\psi\in L^{2}(\mathbb{R})$ such that the wavelet system $\{\psi_{j,k} : k,j\in \mathbb{Z}\}$ is orthonormal and satisfies the ``discrete Calder\'{o}n condition'', $\sum\limits_{j\in \mathbb{Z}}|\widehat{\psi}(2^{j}\xi)|^{2}=1$ for a.e. $\xi\in \mathbb{R}$, then it is complete in $L^{2}(\mathbb{R}),$ where $\psi_{j,k}(x)=2^{\frac{j}{2}}\psi(2^{j}x-k)$, $j,k\in \mathbb{Z}.$ This was conjectured by Guido Weiss. (Infact, the conjecture is that the ``discrete Calder\'{o}n condition'' is both necessary as well as sufficient for the completion of an orthonormal wavelet system). Yves Meyer claimed that the completion of a wavelet system is equivalent to the discrete Calder\'{o}n condition and $\sum\limits_{j=0}^{\infty}\widehat{\psi}(2^{j}\xi)\overline{\widehat{\psi}(2^{j}(\xi+q))}=0,$ which holds for a.e. $\xi\in \mathbb{R}$ and all $q\in 2\mathbb{Z}+1$. Later, these questions were investigated by several authors in various settings. For the history and the references, we refer to \cite{zethesis}.\\
	
	 In 1995, G. Gripenberg \cite{Gr} and X. Wang \cite{wa} independently proved that for $\psi\in L^{2}(\mathbb{R}),$ the system $\{\psi_{j,k} : j,k\in \mathbb{Z}\}$ forms an orthonormal basis for $L^{2}(\mathbb{R})$ iff $\|\psi\|_{L^{2}(\mathbb{R})}=1,$ $\psi$ satisfies discrete Calder\'{o}n condition and $\sum\limits_{j=0}^{\infty}\widehat{\psi}(2^{j}\xi)\overline{\widehat{\psi}(2^{j}(\xi+q))}=0,$ $q\in 2\mathbb{Z}+1$, for a.e. $\xi\in \mathbb{R}$. In 2001, Rzeszotnik \cite{ze}, proved that the discrete Calder\'{o}n condition characterizes the completeness of an orthonormal wavelet basis. For a study of wavelets in $L^{2}(\mathbb{R})$, we refer to \cite{weiss}\\
	
	The aim of this paper is to obtain an analogue of the discrete Calder\'{o}n condition in $L^{2}(\mathbb{R}^{2})$ for a system of twisted wavelets and investigate whether discrete Calder\'{o}n condition is sufficient to prove the completeness of an orthonormal twisted wavelet system in $L^{2}(\mathbb{R}^{2})$. The twisted convolution is a non-standard convolution which arises while transferring the convolution of the Heisenberg group to the complex plane. Under this operation of twisted convolution, $L^{1}(\mathbb{R}^{2})$ turns out to be a non commutative Banach algebra. Hence the study of (twisted) shift-invariant spaces on $\mathbb{R}^{2}$ completely differs from the perspective of the usual shift-invariant spaces on $\mathbb{R}$.\\
	
	The twisted wavelet system is formed using the usual dilations and twisted translations in $L^{2}(\mathbb{R}^{2})$. More precisely, the twisted wavelet system is defined to be $\{D_{2^{j}}(T^{t}_{(k,l)})^{2^{-2j}}\psi_{j} : k,l,j\in \mathbb{Z}\},$ where $D_{a}$ denotes the usual dilation in $L^{2}(\mathbb{R}^{2})$ and $(T^{t}_{(k,l)})^{\lambda}$ denotes the $\lambda-$twisted translation as defined in Definition \ref{lamtwistdef} in this paper. Notice that instead of taking a single function $\psi$, we take a collection of functions $\{\psi_{j}\}_{j\in \mathbb{Z}}$. This is due to the following fact. In the case of the Heisenberg group $\mathbb{H},$ the wavelet system is defined using the left translates and the non-isotropic dilations on the Heisenberg group (see \cite{shannon}). In other words, the wavelet system is defined to be $\{\delta_{2^{j}}L_{\gamma}\psi\}_{j,\gamma}.$ Here $L_{\gamma}$ denotes the left translate by $\gamma$ belonging to a lattice $\Gamma$ in $\mathbb{H}$ and for $a>0$, $\delta_{a}$ denotes the non-isotropic dilation defined by $\delta_{a}\psi(x,y,t)=|a|^{2}\psi(ax,ay,a^{2}t),~~(x,y,t)\in \mathbb{H}$. But a standard lattice in $\mathbb{H}$ is given by $\{(2k,l,m) : k,l,m\in \mathbb{Z}\}.$ By transferring this structure of wavelets to the complex plane, we obtain the system $\{D_{2^{j}}(T^{t}_{(k,l)})^{2^{-2j}}\psi_{j} : k,l,j\in \mathbb{Z}\}$. (See Lemma 3.6 in \cite{arati_orthonormality}).\\
		
	 Since the aim of this paper deals with orthonormal twisted wavelet system, we first provide an example in section 3. Then we obtain an analogue of the discrete Calder\'{o}n condition for this system, which turns out to be $\sum\limits_{j\in\mathbb{Z}}2^{j}\|K^{2^{-2j}}_{\psi_{j}}(\cdot,2^{j}\eta)\|^{2}=1$, for a.e. $\eta\in \mathbb{R}$, where $K_{\psi_{j}}$ denotes the kernel of the Weyl transform of $\psi_{j}$. As mentioned earlier, we wish to explore whether this condition leads to the completion of an orthonormal twisted wavelet system. However, we are able to answer our question affirmatively only for the modified wavelet system defined by $\{D_{2^{j+m}}(T^{t}_{(k,l)})^{2^{-2j}}\psi_{j} : k,l,m\in \mathbb{Z},j\geq 0\}.$ This is due to the following fact. In the classical case, if $W_{j}=\overline{\text{span}\{\psi_{j,k} :k\in \mathbb{Z}\}}$ and $V=\bigg(\bigoplus\limits_{j\geq 0}W_{j}\bigg)^{\perp},$ then one can show that $D_{2^{i}}(V)=\bigg(\bigoplus\limits_{j\geq i}W_{j}\bigg)^{\perp}.$ This property does not hold if we take $W_{j}=\overline{\text{span}\{D_{2^{j}}(T_{(k,l)})^{2^{-2j}}\psi_{j} :k,l\in \mathbb{Z}\}}.$ This is because in the classical case, the translates in a wavelet system are independent of dilation factors. In the case of twisted wavelets, the twisted translates also contain $j$ factor arising from dilations. This problem is resolved if we make use of a modified twisted wavelet system. One can ask why not we define the twisted wavelet system by taking only the operations $D_{2^{j}}T^{t}_{(k,l)}$ and define $W_{j}=\overline{\text{span}}\{D_{2^{j}}T_{(k,l)}\psi :k,l\in \mathbb{Z}\}.$ But $W_{j}$ is not a twisted shift invariant space. In fact, $T^{t}_{(k^{'},l^{'})}D_{2^{j}}T^{t}_{(k,l)}=D_{2^{j}}(T^{t}_{(2^{j}k^{'},2^{j}l^{'})})^{2^{-2j}}T^{t}_{(k,l)},$ which does not belong to $W_{j}$. Thus we end up with the modified twisted wavelet system, defined by $\{D_{2^{j+m}}(T_{(k,l)})^{2^{-2j}}\psi_{j} :k,l,m\in \mathbb{Z},j\geq 0\}$.   
	\section{Preliminaries}
	Let $\h\neq 0$ be a separable Hilbert space.
	\begin{definition}
		A sequence $\{f_{k} : k\in\mathbb{Z}\}$ in $\h$ is called a frame for $\h$ if there exist two constants $A,B>0$ such that
		\begin{align*}
			A\|f\|^{2}\leq \sum_{k\in \mathbb{Z}}|\langle f,f_{k}\rangle|^{2} \leq B\|f\|^{2},\hspace{10mm}\forall\ f\in \h.\numberthis \label{frm}
		\end{align*}
	\end{definition}
	If \eqref{frm} holds with $A=B=1$, then $\{f_{k} : k\in \mathbb{Z}\}$ is called a Parseval frame.\\
	
		We refer to \cite{CN,Heilbook} for the study of frames.\\
		
		The Heisenberg group $\mathbb{H}$ is a nilpotent Lie group whose underlying manifold is $\mathbb{R}\times\mathbb{R}\times\mathbb{R}$ with the group operation defined by $(x,y,t)(u,v,s)=(x+u,y+v,t+s+\frac{1}{2}(u\cdot y-v\cdot x))$ and the Haar measure is the Lebesgue measure $dx\,dy\,dt$ on $\mathbb{R}\times\mathbb{R}\times\mathbb{R}$. Using the Schr\"{o}dinger representation $\pi_{\lambda},$ $\lambda\in\mathbb{R}^{\ast},$ given by
		\begin{equation*}
			\pi_{\lambda}(x,y,t)\p(\xi)=e^{2\pi i\lambda t}e^{2\pi i\lambda(x\cdot\xi+\frac{1}{2}x\cdot y)}\p(\xi+y),~\p\in L^{2}(\mathbb{R}),
		\end{equation*}
		we define the group Fourier transform of $f\in L^{1}(\mathbb{H})$ as
		\begin{equation*}
			\widehat{f}(\lambda)=\int_{\mathbb{H}}f(x,y,t)\pi_{\lambda}(x,y,t)\,dx\, dy\,dt,~\text{where}~\lambda\in\mathbb{R}^{\ast},
		\end{equation*}
		which is a bounded operator on $L^{2}(\mathbb{R})$. In otherwords, for $\p\in L^{2}(\mathbb{R})$, we have
		\begin{equation*}
			\widehat{f}(\lambda)\p=\int_{\mathbb{H}}f(x,y,t)\pi_{\lambda}(x,y,t)\p\,dx\, dy\, dt,
		\end{equation*}
		where the integral is a Bochner integral taking values in $L^{2}(\mathbb{R}).$ If $f$ is also in $L^{2}(\mathbb{H})$, then $\widehat{f}(\lambda)$ is a Hilbert-Schmidt operator. Define
		\begin{equation*}
			f^{\lambda}(x,y)=\int_{\mathbb{R}}e^{2\pi i\lambda t}f(x,y,t)\,dt,
		\end{equation*}
		which is the inverse Fourier transform of $f$ in the $t$ variable. Then we can write
		\begin{equation*}
			\widehat{f}(\lambda)=\int_{\mathbb{R}^{2}}f^{\lambda}(x,y)\pi_{\lambda}(x,y,0)dx\, dy.
		\end{equation*}
		Let $g\in L^{1}(\mathbb{R}^{2})$. Define
		\begin{equation*}
			W_{\lambda}(g)=\int_{\mathbb{R}^{2}}g(x,y)\pi_{\lambda}(x,y,0)\,dx\, dy.
		\end{equation*} The transform $W_{\lambda}(g)$ is an integral operator with kernel \begin{equation*}
		K^{\lambda}_{g}(\xi,\eta)=\int_{\mathbb{R}}g(x,\eta-\xi)e^{\pi i\lambda x\cdot(\eta+\xi)}\, dx,~ \xi,\eta\in \mathbb{R}.
	\end{equation*}
		When $\lambda=1$, it is called the Weyl transform of $g$, denoted by $W(g)$. If $g\in L^{1}\cap L^{2}(\mathbb{R}^{2})$, then $K^{\lambda}_{g}\in L^{2}(\mathbb{R}^{2})$, which implies that $W_{\lambda}(g)$ is a Hilbert-Schmidt operator whose norm is given by $\|W_{\lambda}(g)\|_{\mathcal{B}_{2}}=\|K^{\lambda}_{g}\|_{L^{2}(\mathbb{R}^{2})}=\frac{1}{|\lambda|^{\frac{1}{2}}}\|g\|_{L^{2}(\mathbb{R}^{2})}$, where $\mathcal{B}_{2}$ is the Hilbert space of Hilbert-Schmidt operators on $L^{2}(\mathbb{R})$ with inner product $(T,S)=\tr(TS^{\ast})$.\\
		
		For a detailed study of analysis on the Heisenberg group we refer to \cite{thangavelu,follandphase}.\\
		
\begin{definition}\label{lamtwistdef}
	Let $f\in L^{2}(\mathbb{R}^{2})$ and $\lambda\in \mathbb{R}^{\ast}$. Then, for $k,l\in \mathbb{Z}$, $\lambda-$twisted translation of $f$, denoted by $(T^{t}_{(k,l)})^{\lambda}f$, is defined to be \begin{equation*}
		(T^{t}_{(k,l)})^{\lambda}f(x,y)=e^{\pi i\lambda(x\cdot l-y\cdot k)}f(x-k,y-l).
	\end{equation*}
\end{definition}
		The $\lambda-$ twisted translation satisfies the following property.
	\begin{align*}
		(T_{(k_1,l_1)}^t)^{\lambda}(T_{(k_2,l_2)}^t)^{\lambda}=e^{-\pi i\lambda(k_1\cdot l_2-l_1\cdot k_2)}(T_{(k_1+k_2,l_1+l_2)}^t)^{\lambda},\ \ \ \ \ \forall\ (k_1,l_1),(k_2,l_2)\in\mathbb{Z}^{2}.\numberthis  \label{comptsttrnsl}
	\end{align*}\\
When $\lambda=1$, then $T^{t}_{(k,l)}f$ is called the twisted translation of $f$.\\

For a further study of $\lambda-$translates, we refer to \cite{RHG}\\

A closed subspace $V\subset L^{2}(\mathbb{R}^{2})$ is said to be a twisted shift-invariant space if $f\in V$, then $T^{t}_{(k,l)}f\in V$, for any $k,l\in \mathbb{Z}.$\\

Let $V^{t}(\phi)$ denote the principal twisted shift-invariant space, defined by  $V^{t}(\phi)=\overline{\text{span}\{T^{t}_{(k,l)}\p : k,l\in \mathbb{Z}\}}$.

\begin{theorem}\label{decomposition}\cite{rabeetha_3}
	If $V$ is a twisted shift-invariant subspace in $L^{2}(\mathbb{R}^{2})$, then there exists a collection of functions $\{\phi_{i}\}_{i\in \N}$ in $L^{2}(\mathbb{R}^{2})$ such that $V=\bigoplus\limits_{i\in\N}\vv(\phi_{i})$, where each $\phi_{i}$ is a Parseval frame generator of $\vv(\phi_{i})$, for $i\in\N$.  Moreover, if $f\in V$, then $\|f\|^{2}_{2}=\sum\limits_{i\in \mathbb{N}}\|r_{i}\|^{2}_{L^{2}(\mathbb{T}\times \mathbb{T}; \left[\phi_{i},\phi_{i}\right])},$ where $r_{i}\in L^{2}(\mathbb{T}\times \mathbb{T}; \left[\phi_{i},\phi_{i}\right]).$
\end{theorem}

In \cite{rabeetha_2}, we introduced Zak transform on $ L^{2}(\mathbb{R}^{2})$ associated with the Weyl transform, which is defined as follows : For $\p\in L^{2}(\mathbb{R}^{2})$, \[Z_{W}\p(\xi,\xi^{'},\e)=\sum\limits_{m\in \mathbb{Z}}K_{\p}(m+\xi,\e)e^{-2\pi im\cdot \xi^{'}},\hspace{3mm}\xi,\xi^{'}\in\mathbb{T},\e\in \mathbb{R}.\] Then $Z_{W}$ turns out to be an isometric isomorphism between $L^{2}(\mathbb{R}^{2})$ and $L^{2}(\mathbb{T}\times \mathbb{T}\times \mathbb{R})$. The operator $Z_{W}$ is called the Weyl-Zak transform and it acts on twisted translation in the following way. \[Z_{W}T^{t}_{(k,l)}\p(\xi,\xi^{'},\e)=e^{2\pi i(k\cdot \xi+l\cdot \xi^{'})}e^{\pi ik\cdot l}Z_{W}\p(\xi,\xi^{'},\e),\hspace{3mm}k,l\in \mathbb{Z},\xi,\xi^{'}\in \mathbb{T},\e\in \mathbb{R}.\numberthis\label{zaktwist}\]

By using the Plancherel theorem for the Fourier series, we get the following relation between the Weyl Zak transform $Z_{W}$ of $\phi$ and the kernel of the Weyl transform of $\phi$. 
\begin{equation}\label{zakker}
	\int_{\mathbb{T}}\int_{\mathbb{T}}|Z_{W}\phi(\xi,\xi^{'},\eta)|^{2}\,d\xi d\xi^{'}=\int_{\mathbb{R}}|K_{\phi}(\xi,\eta)|^{2}\, d\xi.
\end{equation}
 The bracket map $\left[\p,\phi\right]$, for $\phi\in L^{2}(\mathbb{R}^{2})$, is defined to be \[\left[\p,\phi\right](\xi,\xi^{'})=\int_{\mathbb{R}}Z_{W} \p(\xi,\xi^{'},\e)\overline{Z_{W} \phi(\xi,\xi^{'},\e)}\,d\e.\] which satisfies \[\|\left[\p,\phi\right]\|_{L^{1}(\mathbb{T}\times\mathbb{T})}\leq \|\p\|_{L^{2}(\mathbb{R}^{2})}\|\phi\|_{L^{2}(\mathbb{R}^{2})}.\] We refer \cite{rabeetha_2} for further details. We make use of the following results proved in \cite{rabeetha_2} in the sequel.

\begin{proposition}\label{propo}
	Let $\p\in L^{2}(\mathbb{R}^{2})$. Then $f\in V^{t}(\p)$ if and only if
	\begin{equation*}
		Z_{W}f(\xi,\xi^{'},\e)=r(\xi,\xi^{'})Z_{W}\p(\xi,\xi^{'},\e),~\text{for a.e}~\xi,\xi^{'}\in \mathbb{T}, \e\in \mathbb{R},
	\end{equation*}
	for some $r\in L^{2}(\mathbb{T} \times \mathbb{T};\left[\p,\p\right])$, where $\left[\p,\p\right](\xi,\xi^{'})=\int_{\mathbb{R}}|Z_{W}\phi(\xi,\xi^{'},\eta)|^{2}\, d\eta.$
\end{proposition}

\begin{theorem}\label{par}
		Let $\p\in L^{2}(\mathbb{R}^{2})$. Then $\{T^{t}_{(k,l)}\phi : k,l\in \mathbb{Z}\}$ is a frame sequence with frame bounds $A,B>0$ iff \begin{equation*}
			0<A\leq \left[\phi,\phi\right](\xi,\xi^{'})\leq B<\infty,~\text{for a.e.}~(\xi,\xi^{'})\in \Omega_{\phi},
		\end{equation*}
	where $\Omega_{\phi}=\{(\xi,\xi^{'})\in \mathbb{T}\times \mathbb{T} :\left[\phi,\phi\right](\xi,\xi^{'})\neq 0\}.$
\end{theorem}

For $a>0$, we define the dilation $D_{a}$ on $L^{2}(\mathbb{R}^{2})$ as $D_{a}\phi(x,y)=a\phi(ax,ay)$, $(x,y)\in \mathbb{R}^{2},$ $\phi\in L^{2}(\mathbb{R}^{2}).$\\

\begin{lemma}\label{dilatwisker}\cite{arati_orthonormality}
	Let $\p \in L^{2}(\mathbb{R}^{2})$. Then the kernel of $W_{\lambda}(D_{2^{j}}\p)$ and hence that of $W_{\lambda}(D_{2^{j}}(T^{t}_{(2k,l)})^{\lambda2^{-2j}}\p)$, denoted by $K^{\lambda}_{D_{2^{j}}\p}$ and $K^{\lambda}_{D_{2^{j}}(T^{t}_{(2k,l)})^{\lambda2^{-2j}}\p}$ respectively, satisfy the following relations.
	\begin{equation*}
		K^{\lambda}_{D_{2^{j}}\p}(\xi,\eta)=K^{\lambda2^{-2j}}_{\phi}(2^{j}\xi,2^{j}\eta),
	\end{equation*}
	\begin{equation*}
		K^{\lambda}_{D_{2^{j}}(T^{t}_{(2k,l)})^{\lambda2^{-2j}}\p}=e^{2\pi i\lambda2^{-2j}k(2^{j+1}\xi+l)}K^{\lambda2^{-2j}}_{\p}(2^{j}\xi+l,2^{j}\eta).
	\end{equation*}
\end{lemma}

\section{The Main result}
Before moving to the main result, let us look at an example of an orthonormal twisted wavelet system.
\begin{example}\label{twavex}
For $j\in \mathbb{Z},$ let \begin{equation*}
	\psi_{j}(x,y)=e^{\pi ixy2^{-2j}}\chi^{H}_{[0,1]}(x)\chi^{H}_{[0,1]}(y),~~x,y\in \mathbb{R},
\end{equation*} where $\chi^{H}_{[0,1]}$ is the Haar wavelet in $L^{2}(\mathbb{R})$ given by $\chi^{H}_{[0,1]}=\begin{cases}
	1, & 0\leq x\leq \frac{1}{2}\\
	-1, & \frac{1}{2}< x\leq 1\\
	0, &\text{otherwise}~~~.
\end{cases}$\\ Then we show that the collection $\{D_{2^{j}}\a \psi_{j}: k,l,j\in \mathbb{Z}\}$ is an orthonormal system in $L^{2}(\mathbb{R}^{2})$. In fact, \begin{align*}
		D_{2^{j}}\a \psi_{j}(x,y)&=2^{j}\a\psi_{j}(2^{j}x,2^{j}y)\\
		&=2^{j}e^{\pi i(2^{j}xl-2^{j}yk)2^{-2j}}\psi_{j}(2^{j}x-k,2^{j}y-l)\\
		&=2^{j}e^{\pi i(2^{j}xl-2^{j}yk)2^{-2j}}e^{\pi i(2^{j}x-k)(2^{j}y-l)2^{-2j}}\chi^{H}_{[0,1]}(2^{j}x-k)\chi^{H}_{[0,1]}(2^{j}y-l)\\
		&=2^{j}e^{\pi i(2^{j}xl-2^{j}yk)2^{-2j}}e^{\pi i(2^{2j}xy-2^{j}xl-2^{j}yk+kl)2^{-2j}}\chi^{H}_{[0,1]}(2^{j}x-k)\chi^{H}_{[0,1]}(2^{j}y-l)\\
		&=2^{j}e^{\pi i(2^{2j}xy-2^{j+1}yk+kl)2^{-2j}}\chi^{H}_{[0,1]}(2^{j}x-k)\chi^{H}_{[0,1]}(2^{j}y-l).
	\end{align*}
Consider, for $j,k,l,j^{'},k^{'},l^{'}\in \mathbb{Z}$,
\begin{align*}
	\langle D_{2^{j}}\a \psi_{j}, D_{2^{j{'}}}(T_{k^{'},l^{'}})^{2^{-2j^{'}}}\psi_{j^{'}}\rangle
	&=\int_{\mathbb{R}}\int_{\mathbb{R}}D_{2^{j}}\a \psi_{j}(x,y)\overline{ D_{2^{j{'}}}(T_{k^{'},l^{'}})^{2^{-2j^{'}}}\psi_{j^{'}}(x,y)}\, dxdy\\
	&=\int_{\mathbb{R}}\int_{\mathbb{R}}2^{j+j^{'}}e^{\pi i(xy-2^{-(j-1)}yk+kl2^{-2j})}\chi^{H}_{[0,1]}(2^{j}x-k)\chi^{H}_{[0,1]}(2^{j}y-l)\\
	&\hspace{2cm}e^{-\pi i(xy-2^{-(j^{'}-1)}yk^{'}+k^{'}l^{'}2^{-2j^{'}})}\chi^{H}_{[0,1]}(2^{j^{'}}x-k^{'})\chi^{H}_{[0,1]}(2^{j^{'}}y-l^{'})\\
	&=e^{\pi i(kl2^{-2j}-k^{'}l^{'}2^{-2j^{'}})}\bigg(\int_{\mathbb{R}}2^{\frac{j+j^{'}}{2}}\chi^{H}_{[0,1]}(2^{j}x-k)\chi^{H}_{[0,1]}(2^{j^{'}}x-k^{'})\,dx\bigg)\\
	&\hspace{1cm}\bigg(\int_{\mathbb{R}}e^{2\pi iy(2^{-j^{'}}k^{'}-2^{-j}k)}2^{\frac{j+j^{'}}{2}}\chi^{H}_{[0,1]}(2^{j}y-l)\chi^{H}_{[0,1]}(2^{j^{'}}y-l^{'})\,dy\bigg).
\end{align*}
Since the system $\{2^{\frac{j}{2}}\chi^{H}_{[0,1]}(2^{j}x-k) : k,j\in \mathbb{Z}\}$ forms an orthonormal basis for $L^{2}(\mathbb{R})$, the first integral turns out to be zero, if $k\neq k^{'}$ or $j\neq j^{'}$. Hence, we get  
\begin{align*}
	\langle D_{2^{j}}\a \psi_{j},& D_{2^{j{'}}}(T_{k^{'},l^{'}})^{2^{-2j^{'}}}\psi_{j^{'}}\rangle\\
	&=\begin{cases}
		0 &~~\text{if}~~k\neq k^{'}~\text{or}~j\neq j^{'}\\
		e^{\pi i2^{-2j}k(l-l^{'})}\int\limits_{\mathbb{R}}2^{j}\chi^{H}_{[0,1]}(2^{j}y-l^{'})\chi^{H}_{[0,1]}(2^{j}y-l^{'})\,dy & ~~\text{otherwise}.
	\end{cases}
\end{align*}
Now for $k=k^{'}$ and $j=j^{'}$, we have 
\begin{equation*}
	\langle D_{2^{j}}\a \psi_{j}, D_{2^{j}}(T_{k,l^{'}})^{2^{-2j}}\psi_{j}\rangle=\begin{cases}
		0 &~~\text{if}~~l\neq l^{'}\\
		1 &~~\text{otherwise}.
	\end{cases}
\end{equation*}
Therefore the system $\{D_{2^{j}}\a \psi_{j}: k,l,j\in \mathbb{Z}\}$ is an orthonormal system in $L^{2}(\mathbb{R}^{2})$.
\end{example}

The natural question is to find an appropriate sufficient condition under which an orthonormal system in $L^{2}(\mathbb{R}^{2})$ consisting of twisted wavelets will turn out to be an orthonormal basis for $L^{2}(\mathbb{R}^{2}).$ As in the classical case, we wish to get an ``analogue'' of discrete Calder\'{o}n condition which will do the needful.\\

Let $\psi\in L^{2}(\mathbb{R}^{2})$. Consider the twisted shift-invariant space $V^{t}(\psi)$. We first intend to define a spectral function $\s$ associated with $V^{t}(\psi)$. Suppose $\{T^{t}_{(k,l)}\psi : k,l\in \mathbb{Z}\}$ forms a Parseval frame for $V^{t}(\psi)$. In this case, we define $$\sigma_{\vv(\psi)}(\e):=\int\limits_{\mathbb{R}}|K_{\psi}(\xi,\e)|^{2}\, d\xi.$$ 
More generally, for $\psi\in L^{2}(\mathbb{R}^{2})$, we define $\s$ associated with $V^{t}(\psi)$ as follows.
\begin{equation*}
	\sigma_{\vv(\psi)}(\xi,\xi^{'},\eta):=\begin{cases}
		\dfrac{\int\limits_{\mathbb{R}}|K_{\psi}(q,\e)|^{2}\, dq}{\left[\psi,\psi\right](\xi,\xi^{'})} & ~\text{if}~~(\xi,\xi^{'})\in \Omega_{\psi}\\
		0 &~~~~~~~\text{otherwise}.
	\end{cases}
\end{equation*}
When $\{T^{t}_{(k,l)}\psi : k,l\in \mathbb{Z}\}$ is a Parseval frame for $V^{t}(\psi)$, $\left[\psi,\psi\right](\xi,\xi^{'})=1$ a.e. $(\xi,\xi^{'})\in \Omega_{\psi}$, where $\Omega_{\psi}=\{(\xi,\xi^{'})\in \mathbb{T}\times \mathbb{T} :\left[\psi,\psi\right](\xi,\xi^{'})\neq 0\}.$ (See Theorem \ref{par}). We shall extend this definition to a general twisted shift-invariant subspace in $L^{2}(\mathbb{R}^{2})$.\\

 Let $V$ be a twisted shift invariant subspace of $L^{2}(\mathbb{R}^{2})$. Then $V$ can be decomposed as $V=\bigoplus\limits_{n\in \N}\vv(\psi_{n})$, where $\psi_{n}$ is a Parseval frame generator for $V^{t}(\psi_{n})$ (see Theorem \ref{decomposition}). Then it follows that $\s_{V}(\e)=\sum\limits_{n\in \N}\s_{\vv(\psi_{n})}(\e)$ defined a.e. $\e\in \mathbb{R}$.
\begin{proposition}\label{sigmafini}
	Let $V$ be a twisted shift invariant space in $L^{2}(\mathbb{R}^{2})$ and $\s_{V}\in L^{1}(\mathbb{R})$. Then $$\bigcap_{j\in \mathbb{Z}}D_{2^{j}}(V)=\{0\},$$ where $D_{2^{j}}(V)=\{D_{2^{j}}f : f\in V\}$.
\end{proposition}
\begin{proof}
	Suppose $0\neq f\in \bigcap_{j\in \mathbb{Z}}D_{2^{j}}(V)$. Without loss of generality assume that $\|f\|_{L^{2}(\mathbb{R}^{2})}=1$. Since $f\in \bigcap_{j\geq 0}D_{2^{j}}(V)$, we have $\dj f\in V,$ for all $j\geq 0$. By Theorem \ref{decomposition}, there exists a sequence of functions $\{f_{n}^{j}\}_{n\in \mathbb{N}}$ in $V^{t}(\psi_{n})$ such that $\dj f= \sum\limits f_{n}^{j}.$ By applying Weyl-Zak transform in the above equation, we obtain 
	\begin{equation}\label{zak1}
		Z_{W}\dj f(\xi,\xi^{'},\eta)=\sum\limits_{n\in \N}Z_{W} f_{n}^{j}(\xi,\xi^{'},\e)
	\end{equation} By using Proposition \ref{propo}, there exists $r_{n}^{j}\in L^{2}(\mathbb{T}^{2})$ such that 
\begin{equation}\label{zak2}
	Z_{W} f_{n}^{j}(\xi,\xi^{'},\e)=r_{n}^{j}(\xi,\xi^{'})Z_{W} \psi_{n}(\xi,\xi^{'},\e),
\end{equation} for a.e. $\xi,\xi^{'}\in \mathbb{T}$, $\e\in \mathbb{R}.$ By substituting \eqref{zak2} in \eqref{zak1}, we get \begin{equation}\label{zaksum}
Z_{W}\dj f(\xi,\xi^{'},\eta)=\sum\limits_{n\in \N}r_{n}^{j}(\xi,\xi^{'})Z_{W} \psi_{n}(\xi,\xi^{'},\e).
\end{equation} Moreover, \begin{equation}\label{rnorm}
\sum\limits_{n\in \N}\|r_{n}^{j}\|^{2}_{L^{2}(\mathbb{T}^{2})}=\|\dj f\|^{2}_{L^{2}(\mathbb{R}^{2})}=\| f\|^{2}_{L^{2}(\mathbb{R}^{2})}=1.
\end{equation} (See Proposition \ref{decomposition}). The kernel of the Weyl transform of $\dj f$ is given by $$K_{\dj f}(\xi,\eta)=K_{f}^{2^{2j}}(2^{-j}\xi,2^{-j}\eta).$$ (See Lemma \ref{dilatwisker}). Therefore, \eqref{zaksum} becomes $$\sum\limits_{m\in \mathbb{Z}}K^{2^{2j}}_{f}(2^{-j}(\xi+m),2^{-j}\eta)e^{-2\pi im\xi^{'}}=\sum\limits_{n\in \N}r_{n}^{j}(\xi,\xi^{'})Z_{W} \psi_{n}(\xi,\xi^{'},\e),$$ for a.e. $\xi,\xi^{'}\in \mathbb{T}$, $\e\in \mathbb{R}$, which implies that $$\sum\limits_{m\in \mathbb{Z}}K^{2^{2j}}_{f}(2^{-j}(\xi+m),\eta)e^{-2\pi im\xi^{'}}=\sum\limits_{n\in \N}r_{n}^{j}(\xi,\xi^{'})Z_{W} \psi_{n}(\xi,\xi^{'},2^{j}\e),$$ for a.e. $\xi,\xi^{'}\in \mathbb{T}$, $\e\in \mathbb{R}$. Now, consider 
\begin{align*}
\int_{\mathbb{T}^{2}}\int_{1}^{2}\bigg|\sum\limits_{m\in \mathbb{Z}}K^{2^{2j}}_{f}&(2^{-j}(\xi+m),\eta)e^{-2\pi im\xi^{'}}\bigg|\,d\eta d\xi d\xi^{'}\\
&=\int_{\mathbb{T}^{2}}\int_{1}^{2}\bigg|\sum\limits_{n\in \N}r_{n}^{j}(\xi,\xi^{'})Z_{W} \psi_{n}(\xi,\xi^{'},2^{j}\e)\bigg|\,d\eta d\xi d\xi^{'}\\
&\leq\int_{\mathbb{T}^{2}}\int_{1}^{2}\bigg(\sum\limits_{n\in \N}\big|r_{n}^{j}(\xi,\xi^{'})\big|^{2}\bigg)^{\frac{1}{2}}\bigg(\sum\limits_{n\in \N}|Z_{W} \psi_{n}(\xi,\xi^{'},2^{j}\e)\big|^{2}\bigg)^{\frac{1}{2}}\,d\eta d\xi d\xi^{'}\\	
&=\int_{\mathbb{T}^{2}}\bigg(\sum\limits_{n\in \N}\big|r_{n}^{j}(\xi,\xi^{'})\big|^{2}\bigg)^{\frac{1}{2}}\int_{1}^{2}1 \cdot\bigg(\sum\limits_{n\in \N}|Z_{W} \psi_{n}(\xi,\xi^{'},2^{j}\e)\big|^{2}\bigg)^{\frac{1}{2}}\,d\eta d\xi d\xi^{'}\\
&\leq\int_{\mathbb{T}^{2}}\bigg(\sum\limits_{n\in \N}\big|r_{n}^{j}(\xi,\xi^{'})\big|^{2}\bigg)^{\frac{1}{2}}\bigg(\int_{1}^{2}\sum\limits_{n\in \N}|Z_{W} \psi_{n}(\xi,\xi^{'},2^{j}\e)\big|^{2}\,d\eta\bigg)^{\frac{1}{2}}\,d\xi d\xi^{'}\\
&\leq\bigg(\int_{\mathbb{T}^{2}}\sum\limits_{n\in \N}\big|r_{n}^{j}(\xi,\xi^{'})\big|^{2}\,d\xi d\xi^{'}\bigg)^{\frac{1}{2}}\bigg(\int_{\mathbb{T}^{2}}\int_{1}^{2}\sum\limits_{n\in \N}|Z_{W} \psi_{n}(\xi,\xi^{'},2^{j}\e)\big|^{2}\,d\eta\,d\xi d\xi^{'}\bigg)^{\frac{1}{2}},
\end{align*}
by applying Cauchy-Schwartz inequality thrice. Thus
\begin{align*}
	\int_{\mathbb{T}^{2}}\int_{1}^{2}\bigg|\sum\limits_{m\in \mathbb{Z}}K^{2^{2j}}_{f}(2^{-j}(\xi+m),\eta)e^{-2\pi im\xi^{'}}\bigg|&\,d\eta d\xi d\xi^{'}\\
&\leq 2^{\frac{-j}{2}}\bigg(\int_{\mathbb{T}^{2}}\int_{2^{j}}^{2^{j+1}}\sum\limits_{n\in \N}|Z_{W} \psi_{n}(\xi,\xi^{'},\e)\big|^{2}\,d\eta\,d\xi d\xi^{'}\bigg)^{\frac{1}{2}}\\
&= 2^{\frac{-j}{2}}\bigg(\int_{\mathbb{R}}\chi_{[2^{j},2^{j+1})}(\e)\sum\limits_{n\in \N}\int_{\mathbb{T}^{2}}|Z_{W} \psi_{n}(\xi,\xi^{'},\e)\big|^{2}\,d\xi d\xi^{'}\,d\eta \bigg)^{\frac{1}{2}}\\
&=2^{\frac{-j}{2}}\bigg(\int_{\mathbb{R}}\chi_{[2^{j},2^{j+1})}(\e)\sum\limits_{n\in \N}\int_{\mathbb{R}}|K_{\psi_{n}}(\xi,\eta)|^{2}\, d\xi d\eta\bigg)^{\frac{1}{2}}\\
&=2^{\frac{-j}{2}} \bigg(\int_{\mathbb{R}}\chi_{[2^{j},2^{j+1})}(\e)\sigma_{V}(\eta)\,d\eta\bigg)^{\frac{1}{2}},
\end{align*}
using \eqref{zakker}.
Therefore, we have \begin{equation}\label{mainineq}
	\int_{\mathbb{T}^{2}}\int_{1}^{2}2^{\frac{j}{2}}\bigg|\sum\limits_{m\in \mathbb{Z}}K^{2^{2j}}_{f}(2^{-j}(\xi+m),\eta)e^{-2\pi im\xi^{'}}\bigg|\,d\eta d\xi d\xi^{'}\leq  \bigg(\int_{\mathbb{R}}\chi_{[2^{j},2^{j+1})}(\e)\sigma_{V}(\eta)\,d\eta\bigg)^{\frac{1}{2}}.
\end{equation}
Since $\sigma_{V}\in L^{1}(\mathbb{R})$, we can apply Lebesgue dominated convergence theorem for the sequence of functions $\{\chi_{[2^{j},2^{j+1})}\sigma_{V}\}_{j\in \mathbb{Z}}$, from which it follows that $$\int_{\mathbb{T}^{2}}\int_{1}^{2}\lim\limits_{j\to \infty}2^{\frac{j}{2}}\bigg|\sum\limits_{m\in \mathbb{Z}}K^{2^{2j}}_{f}(2^{-j}(\xi+m),\eta)e^{-2\pi im\xi^{'}}\bigg|\,d\eta d\xi d\xi^{'}=0.$$ This leads to the fact that 
\begin{equation*}
\int_{\mathbb{T}^{2}}\int_{1}^{2}\lim\limits_{j\to \infty}2^{j}\bigg|\sum\limits_{m\in \mathbb{Z}}K^{2^{2j}}_{f}(2^{-j}(\xi+m),\eta)e^{-2\pi im\xi^{'}}\bigg|^{2}\,d\eta d\xi d\xi^{'}=0. 
\end{equation*}
For each $i\in \mathbb{Z}$, we have $D_{2^{-i}}f\in \bigcap\limits_{j\geq 0}D_{2^{j}}(V)$ and $\|D_{2^{-i}}f\|_{L^{2}(\mathbb{R}^{2})}=1.$ Thus, 
\begin{align}\label{i}
	0&=\int_{\mathbb{T}^{2}}\int_{1}^{2}\lim\limits_{j\to \infty}2^{j}\bigg|\sum\limits_{m\in \mathbb{Z}}K^{2^{2j}}_{D^{2^{-i}}f}(2^{-j}(\xi+m),\eta)e^{-2\pi im\xi^{'}}\bigg|^{2}\,d\eta d\xi d\xi^{'}\nonumber\\
	&=\int_{\mathbb{T}^{2}}\int_{1}^{2}\lim\limits_{j\to \infty}2^{j}\bigg|\sum\limits_{m\in \mathbb{Z}}K^{2^{2(j+i)}}_{f}(2^{-(j+i)}(\xi+m),2^{-i}\eta)e^{-2\pi im\xi^{'}}\bigg|^{2}\,d\eta d\xi d\xi^{'}\nonumber\\
	&=\int_{\mathbb{T}^{2}}\int_{2^{i}}^{2^{i+1}}\lim\limits_{j\to \infty}2^{j+i}\bigg|\sum\limits_{m\in \mathbb{Z}}K^{2^{2(j+i)}}_{f}(2^{-(j+i)}(\xi+m),\eta)e^{-2\pi im\xi^{'}}\bigg|^{2}\,d\eta d\xi d\xi^{'}\nonumber\\
	&=\int_{\mathbb{T}^{2}}\int_{2^{i}}^{2^{i+1}}\lim\limits_{j\to \infty}2^{j}\bigg|\sum\limits_{m\in \mathbb{Z}}K^{2^{2j}}_{f}(2^{-j}(\xi+m),\eta)e^{-2\pi im\xi^{'}}\bigg|^{2}\,d\eta d\xi d\xi^{'},~~\forall~~ i\in \mathbb{Z}.
\end{align}
We observe that \begin{align}\label{combine}
	\int_{\mathbb{T}^{2}}\int_{0}^{\infty}\lim\limits_{j\to \infty}2^{j}\bigg|\sum\limits_{m\in \mathbb{Z}}K^{2^{2j}}_{f}&(2^{-j}(\xi+m),\eta)e^{-2\pi im\xi^{'}}\bigg|^{2}\,d\eta d\xi d\xi^{'}\nonumber\\
	&=\sum\limits_{i\in\mathbb{Z}}\int_{\mathbb{T}^{2}}\int_{2^{i}}^{2^{i+1}}\lim\limits_{j\to \infty}2^{j}\bigg|\sum\limits_{m\in \mathbb{Z}}K^{2^{2j}}_{f}(2^{-j}(\xi+m),\eta)e^{-2\pi im\xi^{'}}\bigg|^{2}\,d\eta d\xi d\xi^{'}.
\end{align} By combining \eqref{i} and \eqref{combine}, we get
$$\int_{\mathbb{T}^{2}}\int_{0}^{\infty}\lim\limits_{j\to \infty}2^{j}\bigg|\sum\limits_{m\in \mathbb{Z}}K^{2^{2j}}_{f}(2^{-j}(\xi+m),\eta)e^{-2\pi im\xi^{'}}\bigg|^{2}\,d\eta d\xi d\xi^{'}=0,~\forall~~~f\in \bigcap\limits_{j\in \mathbb{Z}}D_{2^{j}}(V).$$ By proceeding with the above argument with respect to the interval $[-2,-1]$ instead of $[1,2]$, we can show that $$\int_{\mathbb{T}^{2}}\int_{-\infty}^{0}\lim\limits_{j\to \infty}2^{j}\bigg|\sum\limits_{m\in \mathbb{Z}}K^{2^{2j}}_{f}(2^{-j}(\xi+m),\eta)e^{-2\pi im\xi^{'}}\bigg|^{2}\,d\eta d\xi d\xi^{'}=0,~\forall~~~f\in \bigcap\limits_{j\in \mathbb{Z}}D_{2^{j}}(V).$$ Thus,
$$\int_{\mathbb{T}^{2}}\int_{\mathbb{R}}\lim\limits_{j\to \infty}2^{j}\bigg|\sum\limits_{m\in \mathbb{Z}}K^{2^{2j}}_{f}(2^{-j}(\xi+m),\eta)e^{-2\pi im\xi^{'}}\bigg|^{2}\,d\eta d\xi d\xi^{'}=0$$
Now, by using the Plancherel theorem for the Fourier series, we have
     \begin{align*}
     	0&=\int_{\mathbb{T}^{2}}\int_{\mathbb{R}}\lim\limits_{j\to \infty}2^{j}\bigg|\sum\limits_{m\in \mathbb{Z}}K^{2^{2j}}_{f}(2^{-j}(\xi+m),\eta)e^{-2\pi im\xi^{'}}\bigg|^{2}\,d\eta d\xi d\xi^{'}\\
     	&=\int_{\mathbb{R}}\int_{\mathbb{T}}\lim\limits_{j\to \infty}2^{j}\sum\limits_{m\in \mathbb{Z}}\big|K^{2^{2j}}_{f}(2^{-j}(\xi+m),\eta)\big|^{2}\,d\eta d\xi\\
     	&=\int_{\mathbb{R}}\int_{\mathbb{R}}\lim\limits_{j\to \infty}2^{j}\big|K^{2^{2j}}_{f}(2^{-j}\xi,\eta)\big|^{2}\,d\eta d\xi\\
     	&=\int_{\mathbb{R}}\int_{\mathbb{R}}\lim\limits_{j\to \infty}2^{2j}\big|K^{2^{2j}}_{f}(\xi,\eta)\big|^{2}\,d\eta d\xi\\
     	&=\lim\limits_{j\to \infty}2^{2j}\int_{\mathbb{R}}\int_{\mathbb{R}}\big|K^{2^{2j}}_{f}(\xi,\eta)\big|^{2}\,d\eta d\xi\\
     	&=\lim\limits_{j\to \infty}\|f\|^{2}_{L^{2}(\mathbb{R}^{2})}=\|f\|^{2}_{L^{2}(\mathbb{R}^{2})}=1,
     \end{align*}
which is a contradiction. This completes the proof. 
\end{proof}
Now we shall show that if the modified twisted wavelet system is orthonormal and satisfies the discrete Calder\'{o}n condition, then it will turn out to be an orthonormal basis for $L^{2}(\mathbb{R}^{2})$. 
\begin{theorem}
	Let $\psi_{j}\in L^{2}(\mathbb{R}^{2})$ with $\|\psi_{j}\|_{L^{2}(\mathbb{R}^{2})}=1$, for all $j\in \mathbb{Z}$. Suppose the collection $\{\D\a \psi_{j} : k,l,m\in \mathbb{Z},~j\geq 0\}$ is an orthonormal system in $L^{2}(\mathbb{R}^{2})$ which satisfies $\sum\limits_{j\in\mathbb{Z}}2^{j}\|K^{2^{-2j}}_{\psi_{j}}(\cdot,2^{j}\eta)\|^{2}=1$, for a.e. $\eta\in \mathbb{R}$. Then the collection $\{\D\a\psi_{j} : k,l,m\in \mathbb{Z},j\geq 0\}$ forms an orthonormal basis for $L^{2}(\mathbb{R}^{2})$.  
\end{theorem}
\begin{proof}
	Let us define $W_{m,j}:=\overline{\text{span}\{\D\a\psi_{j} : k,l\in \mathbb{Z}\}}$, for $m\in \mathbb{Z}$. 
	Now, for $k,l\in\mathbb{Z}$
	\begin{align}\label{Wjtwist}
		T^{t}_{(k^{'},l^{'})}D_{2^{j}}\a\psi_{j}(x,y)&= e^{\pi i(xl^{'}-yk^{'})}D_{2^{j}}\a\psi_{j}(x-k^{'},y-l^{'})\nonumber\\
		&=2^{j}e^{\pi i(xl^{'}-yk^{'})}\a\psi_{j}(2^{j}(x-k^{'}),2^{j}(y-l^{'}))\nonumber\\
		&=2^{j}e^{\pi i(2^{j}x2^{j}l^{'}-2^{j}y2^{j}k^{'})2^{-2j}}\a\psi_{j}(2^{j}(x-k^{'}),2^{j}(y-l^{'}))\nonumber\\
		&=2^{j}(T^{t}_{(2^{j}k^{'},2^{j}l^{'})})^{2^{-2j}}\a\psi_{j}(2^{j}x,2^{j}y)\nonumber\\
		&=D_{2^{j}}(T^{t}_{(2^{j}k^{'},2^{j}l^{'})})^{2^{-2j}}\a\psi_{j}(x,y)\nonumber\\
		&=e^{-\pi i2^{-2j}(2^{j}k^{'}l-2^{j}l^{'}k)}D_{2^{j}}(T_{(2^{j}k^{'}+k,2^{j}l^{'}+l}^t)^{2^{-2j}}\psi_{j}(x,y),
	\end{align}
by using \eqref{comptsttrnsl}. By making use of \eqref{Wjtwist}, we can show that $W_{0,j}$ is twisted shift invariant space. Moreover, notice that the subspaces $W_{m,j}$, $m\in \mathbb{Z},j\geq 0$ are mutually orthogonal. Let $Z$ denote the orthogonal complement of $\bigoplus\limits_{m,j\in\mathbb{Z}}W_{m,j}$ in $L^{2}(\mathbb{R}^{2})$. In order to prove the completeness of the system $\{\D\a\psi_{j} : k,l,m\in \mathbb{Z},j\geq 0\}$, it is enough to show that $Z=\{0\}$. Let $V:=(\bigoplus\limits_{j\geq 0}W_{0,j})^{\perp}$. Since $W_{0,j}$ is a twisted invariant subspace, $j\geq 0$, $V$ is also twisted shift invariant space. Further, $D_{2^{m}}(V)=(\bigoplus\limits_{j\geq 0}W_{m,j})^{\perp},~\forall~~ m\in \mathbb{Z}$. Thus, $Z\subset (\bigoplus\limits_{j\geq 0}W_{m,j})^{\perp}=D_{2^{m}}(V),~\forall~~ m\in \mathbb{Z}$, which implies that $$Z\subset \bigcap\limits_{m\in \mathbb{Z}}D_{2^{m}}(V).$$ Since $L^{2}(\mathbb{R}^{2})=V\bigoplus(\bigoplus\limits_{j\geq 0}W_{0,j})$ it follows that
\begin{equation}\label{sigmasum}
	\sigma_{L^{2}(\mathbb{R}^{2})}=\sigma_{V}+\sum\limits_{j\in \mathbb{Z}}\sigma_{W_{0,j}}.
\end{equation}
Now, our aim is to show that the system $\{T^{t}_{(k^{'},l^{l})}D_{2^{j}}\a\psi_{j} : k^{'},l^{'}\in \mathbb{Z},k,l\in \{0,...,2^{j-1}\}\}$ is a Parseval frame for $W_{0,j}$, $j\geq 0$. For any $k,l\in \mathbb{Z}$ there exist unique $p,q\in\mathbb{Z}$ and $r,s\in \{0,...,2^{j-1}\}$ such that $k=p2^{j}+r$ and $l=q2^{j}+s$.
By using \eqref{Wjtwist}, we have
 \begin{align}\label{Wjpar}
	T^{t}_{(k^{'},l^{'})}D_{2^{j}}\a\psi_{j}&=	T^{t}_{(k^{'},l^{'})}D_{2^{j}} (T^{t}_{(p2^{j}+r,q2^{j}+s)})^{2^{-2j}}\psi_{j}\nonumber\\
	&=e^{-\pi i2^{-j}(ps-qr)}T^{t}_{(k^{'},l^{'})}D_{2^{j}} (T^{t}_{(p2^{j},q2^{j})})^{2^{-2j}}(T^{t}_{(r,s)})^{2^{-2j}}\psi_{j}\nonumber\\
	&=e^{-\pi i2^{-j}(ps-qr)}T^{t}_{(k^{'},l^{'})}T^{t}_{(p,q)}D_{2^{j}}(T^{t}_{(r,s)})^{2^{-2j}}\psi_{j}\nonumber\\
	&=e^{-\pi i2^{-j}(ps-qr)}e^{-\pi i(k^{'}q-l^{'}p)}T^{t}_{(k^{'}+p,l^{'}+q)}D_{2^{j}} (T^{t}_{(r,s)})^{2^{-2j}}\psi_{j}
\end{align}
By making use of \eqref{Wjpar} and the fact that $W_{0,j}$ is a twisted shift invariant space, we can show that \begin{equation}\label{Wjspan}
	W_{0,j}=\overline{\text{span}\{T^{t}_{(k^{'},l^{l})}D_{2^{j}}\a\psi_{j} : k^{'},l^{'}\in \mathbb{Z},k,l\in \{0,...,2^{j-1}\}\}}.
\end{equation} By our assumption, $\{D_{2^{j}}\a \psi_{j} : k,l\in \mathbb{Z}\}$ is an orthonormal basis for $W_{0,j}$. For $f\in W_{0,j}$, using \eqref{comptsttrnsl} and \eqref{Wjtwist}, we get,
\begin{align}\label{Wjine}
	\|f\|^{2}_{L^{2}(\mathbb{R}^{2})}&=\sum\limits_{k,l\in\mathbb{Z}}\big|\langle f, D_{2^{j}}\a\psi_{j}\rangle_{L^{2}(\mathbb{R}^{2})} \big|^{2}\nonumber\\
	&=\sum\limits_{k^{'},l^{'}\in \mathbb{Z}}\sum\limits_{k,l=0}^{2^{j}-1}\big|\langle f, D_{2^{j}}(T^{t}_{(k^{'}2^{j}+k,l^{'}2^{j}+l)})^{2^{-2j}}\psi_{j}\rangle_{L^{2}(\mathbb{R}^{2})} \big|^{2}\nonumber\\
	&=\sum\limits_{k^{'},l^{'}\in \mathbb{Z}}\sum\limits_{k,l=0}^{2^{j}-1}\big|\langle f, e^{-\pi i2^{-j}(k^{'}k-l^{'}l)}D_{2^{j}}(T^{t}_{(k^{'}2^{j},l^{'}2^{j})})^{2^{-2j}}(T^{t}_{(k,l)})^{2^{-2j}}\psi_{j}\rangle_{L^{2}(\mathbb{R}^{2})} \big|^{2}\nonumber\\
	&=\sum\limits_{k^{'},l^{'}\in \mathbb{Z}}\sum\limits_{k,l=0}^{2^{j}-1}\big|\langle f, T^{t}_{(k^{'},l^{'})}D_{2^{j}}(T^{t}_{(k,l)})^{2^{-2j}}\psi_{j}\rangle_{L^{2}(\mathbb{R}^{2})} \big|^{2}.
\end{align}
From \eqref{Wjine} and \eqref{Wjspan}, it follows that the system $\{T^{t}_{(k^{'},l^{'})}D_{2^{j}}\a\psi_{j} : k^{'},l^{'}\in \mathbb{Z},k,l\in \{0,...,2^{j-1}\}\}$ is a Parseval frame for $W_{0,j}$, $j\geq 0$. Now, let us compute $\sigma_{W_{0,j}}$. \begin{align*}
	\sigma_{W_{0,j}}(\eta)&=\sum\limits_{k,l=0}^{2^{j-1}}\sigma_{D_{2^{j}} \a\psi_{j}}(\eta)\\
	&=\sum\limits_{k,l=0}^{2^{j-1}}\int_{\mathbb{R}}|K_{D_{2^{j}}\a\psi_{j}}(\xi,\eta)|^{2}\,d\xi\\
	&=\sum\limits_{k,l=0}^{2^{j-1}}\int_{\mathbb{R}}|K^{2^{-2j}}_{\psi_{j}}(2^{j}\xi+l,2^{j}\eta)|^{2}\,d\xi\\	&=\sum\limits_{k,l=0}^{2^{j-1}}2^{-j}\int_{\mathbb{R}}|K^{2^{-2j}}_{\psi_{j}}(\xi,2^{j}\eta)|^{2}\,d\xi\\
	&=2^{j}\int_{\mathbb{R}}|K^{2^{-2j}}_{\psi_{j}}(\xi,2^{j}\eta)|^{2}\,d\xi.\\
\end{align*}
Using \eqref{sigmasum} and the fact that $\sigma(L^{2}(\mathbb{R}^{2}))=1$, we have $$\sigma_{V}(\eta)=1-\sum\limits_{j\geq 0}2^{j}\int_{\mathbb{R}}|K^{2^{-2j}}_{\psi_{j}}(\xi,2^{j}\eta)|^{2}\,d\xi.$$ By the given hypothesis, $\sum\limits_{j\in\mathbb{Z}}2^{j}\|K^{2^{-2j}}_{\psi_{j}}(\cdot,2^{j}\eta)\|^{2}=1$, for a.e. $\eta\in \mathbb{R}$, from which it follows that \begin{align*}
	\int_{\mathbb{R}}\sigma_{V}(\eta)\,d\eta&=\int_{\mathbb{R}}\bigg(\sum\limits_{j\in \mathbb{Z}}2^{j}\int_{\mathbb{R}}|K^{2^{-2j}}_{\psi_{j}}(\xi,2^{j}\eta)|^{2}\,d\xi-\sum\limits_{j\geq 0}2^{j}\int_{\mathbb{R}}|K^{2^{-2j}}_{\psi_{j}}(\xi,2^{j}\eta)|^{2}\,d\xi\bigg)d\eta\\
	&=\int_{\mathbb{R}}\sum\limits_{j\geq 0}2^{-j}\int_{\mathbb{R}}|K^{2^{2j}}_{\psi_{-j}}(\xi,2^{-j}\eta)|^{2}\,d\xi d\eta\\
	&=\sum\limits_{j\geq 0}\int_{\mathbb{R}}\int_{\mathbb{R}}|K^{2^{2j}}_{\psi_{-j}}(\xi,\eta)|^{2}\,d\xi d\eta\\
	&=\sum\limits_{j\geq 0}\|K^{2^{2j}}_{\psi_{-j}}\|^{2}_{L^{2}(\mathbb{R}^{2})}\\
	&=\sum\limits_{j\geq 0}\frac{1}{4^{j}}\|\psi_{-j}\|^{2}_{L^{2}(\mathbb{R}^{2})}=\sum\limits_{j\geq 0}\frac{1}{4^{j}}<\infty.
\end{align*}
By making use of Proposition \ref{sigmafini}, we get $Z=\{0\}$, which completes the proof.
\end{proof}

\begin{remark}
	There was an attempt in \cite{arati_orthonormality} for finding a discrete Calder\'{o}n type condition and prove the completion of the wavelet orthonormal system on the Heisenberg group. However, the result turned out to be negative. The present result of this paper gives us a hope that one can obtain an analogue of the discrete Calder\'{o}n condition and get an orthonormal wavelet basis for $L^{2}(\mathbb{H})$ in future. We leave this as an open problem to the interested readers. 
\end{remark}

\section*{Acknowledgement}

\textbf{Funding} The author (R .R) thanks NBHM, DAE, India for the research project grant. The author (R .V) thanks University Grants Commission, India for the financial support.\\

\bibliographystyle{amsplain}
\bibliography{twisted}
\end{document}